\documentclass{amsart}
\usepackage{amssymb,amsxtra}
\usepackage{graphicx}
\usepackage{amscd}
\usepackage{amsmath}
\usepackage{amsfonts}
\usepackage{amssymb}
\newtheorem{theorem}{Theorem}[section]
\newtheorem{proposition}[theorem]{Proposition}

\newtheorem{corollary}[theorem]{Corollary}
\theoremstyle{definition}
\newtheorem{definition}[theorem]{Definition}
\newtheorem{example}[theorem]{Example}

\newtheorem{remark}[theorem]{Remark}

\title[$C^*$-algebras and Fell bundles associated to a textile system]{$C^*$-algebras and Fell bundles associated to a textile system}
\author{Valentin Deaconu}
\address{Department of Mathematics\\ University of Nevada\\ Reno NV
89557-0084, USA}
\email[Valentin Deaconu]{vdeaconu@unr.edu}

\thanks
{Research partially supported by a UNR JFR Grant}

\date{\today}
\begin{document}

\begin{abstract}
The notion of  textile system was introduced by M. Nasu in order to analyze  endomorphisms and automorphisms of 
 topological Markov shifts. A textile system is given by two finite directed graphs $G$ and $H$ and two  morphisms $p,q:G\to H$, with some extra properties. It turns out that a textile system determines
a first quadrant two-dimensional shift of finite type, via a collection of Wang tiles, and conversely, any such shift is  conjugate to a textile shift. In the case the morphisms $p$ and $q$  have the path lifting property, we prove that they induce  groupoid morphisms $\pi, \rho:\Gamma(G)\to \Gamma(H)$ between the corresponding \'etale groupoids of $G$ and $H$. 

We define two families ${\mathcal A}(m,n)$ and $\bar{\mathcal A}(m,n)$ of $C^*$-algebras associated to a textile shift, and compute them in specific cases. These are graph algebras, associated to some one-dimensional shifts of finite type constructed from the textile shift. Under extra hypotheses, we also define two  families of  Fell bundles  which encode the complexity of these two-dimensional shifts.
We consider several classes of examples of textile shifts,  including the full shift, the Golden Mean shift and shifts associated to  rank two graphs.
\end{abstract}

\maketitle

\section{Introduction}

\medskip

In dynamics, the time evolution of a physical system is often modeled by the iterates of a single transformation. However, multiple symmetries of some systems lead to the study of the join action of several commuting transformations, where new and deep phenomena occur. 

The classical shift of finite type from symbolic dynamics was studied with powerful tools from linear algebra and matrix theory. The number of period $n$ points, the zeta function and the entropy can all be simply expressed in terms of the $k\times k$ transition matrix $A$. The Bowen-Franks group $BF(A)={\mathbb Z}^k/(I-A){\mathbb Z}^k$ is invariant under flow equivalence, and it was recovered in the K-theory of the Cuntz-Krieger algebra ${\mathcal O}_A$ generated by partial isometries $s_1,s_2,...,s_k$ such that 
\[1=\sum_{i=1}^k s_is_i^*, \;\; s_j^*s_j=\sum_{i=1}^kA(j,i)s_is_i^*\;\;\text{for}\;\; 1\le j\le k.\]
The  algebra ${\mathcal O}_A$ is simple and purely infinite if and only if $A$ is transitive (for every $i,j$ there exists $m$ such that $A^m(i,j)\neq 0$), and $A$ is not a permutation matrix. These $C^*$-algebras can also be understood  as graph algebras, which were  studied and generalized  by several authors, see \cite{R}.

The higher dimensional analogue of a shift of finite type consists in all $d$-dimensional arrays of symbols  from a finite alphabet subject to a finite number of local rules. Such arrays can be shifted in each of the $d$ coordinate directions, giving $d$ commuting transformations. There are also $d$ transition matrices, which in general do not commute. There are deep distinctions between the case $d=1$ and $d\ge 2$: for example, it is easy to describe the space of such arrays in the first case, but there is no general algorithm which will decide, given the set of local rules, whether or not the space of such arrays is empty in the second case.

  Although the general theory of multi-dimensional shifts of finite type is still in a rudimentary stage, there are particular classes where significant progress was made, and where graphs and matrices play a useful role. These include  the class of algebraic subshifts, see \cite {S1, S2}  and  the class of two-dimesional shifts associated to  {\em textile systems}, or to {\em Wang tilings}. For these classes, some of the conjugacy invariants, like entropy (the growth rate of the number of patterns one can see in a square of side $n$), the number of periodic points and the zeta functions were computed. 
   
In the literature, there are some papers relating higher dimensional shifts of finite type and $C^*$-algebras. For example, the particular case of shifts  associated to rank  $d$ graphs was studied  by A. Kumjian, D. Pask and others.
In this case, the translations in the coordinate directions are local homeomorphisms, and there is a canonical \'etale groupoid and a $C^*$-algebra associated to such a graph, which is Morita equivalent to a crossed product of an AF-algebra by the group ${\mathbb Z}^d$. Under some mild conditions, the groupoid is essentially free and the $C^*$-algebra is simple and purely infinite. For more details, see \cite{KP}.
Also, in \cite{PRW1} and \cite {PRW2}, the authors analyze the $C^*$-algebra of rank two graphs whose infinite path spaces are Markov subgroups of $({\mathbb Z}/n{\mathbb Z})^{{\mathbb N}^2}$, like the Ledrappier example, see also \cite{KS} and \cite{LS}. In all these examples, the entropy is zero. The connections between higher dimensional subshifts of finite type and operator algebras remains to be  explored further, and we think that this is a fascinating subject.

In this paper, in an attempt to apply results from operator algebra to  arbitrary two-dimensional shifts of finite type supported in the first quadrant, we construct two families of $C^*$-algebras, defined using some one-dimensional shifts associated to a textile shift as in \cite{MP2}.  The K-theory groups of these algebras provide invariants of the two-dimensional shift. We also construct  groupoid morphisms and families of Fell bundles  associated to some particular textile systems. We consider several examples of textile shifts, related to rank two graphs, to the full shift, to the Golden Mean transition matrices and to cellular automata.

{\bf Acknowledgements}. The author wants to express his gratitude to Alex Kumjian, David Pask and Aidan Sims for helpful discussions.


\section{Textile systems and two-dimensional shifts of finite type}

\medskip

Throughout this paper, we consider  finite directed graphs $G=(G^1,G^0)$, where $G^1$ is the set of edges,  $G^0$ is the set of vertices, and $s,r:G^1\to G^0$ are the source and range maps, which are assumed to be onto.

\definition A textile system (see \cite{N}) is a quadruple $T=(G,H,p,q)$, where $G=(G^1,G^0)$, $H=(H^1,H^0)$ are two finite directed graphs, and $p,q:G\rightarrow H$ are two surjective graph morphisms  such that  $(p(e),q(e),r(e),s(e))\in H^1\times H^1 \times G^0\times G^0$ uniquely determines $e\in G^1$.  We have the following commutative diagram:

\[\begin{array}{ccccc}H^0&\stackrel{p}{\leftarrow}&G^0 &\stackrel{q}{\rightarrow} &H^0\\\uparrow\! r &{}&\uparrow\! r&{}&
\uparrow\! r\\H^1&\stackrel{p}{\leftarrow}&G^1&\stackrel{q}{\rightarrow}&H^1\\
\downarrow\! s&{}&\downarrow\! s&{}&\downarrow\! s\\
H^0&\stackrel{p}{\leftarrow} &G^0 &\stackrel{q}{\rightarrow} &H^0\end{array}\]

The dual textile system $\bar{T}=(\bar{G}, \bar{H}, s,r)$ is obtained by interchanging the pairs of maps $(p,q)$ and $(s,r)$. The new graphs $\bar{G}=(G^1,H^1)$ and $\bar{H}=(G^0,H^0)$ have source and range maps given by $p$ and $q$, and $s,r$ are now graph morphisms. Note that, even if the initial graphs $G$ and $H$ have no sinks, the new graphs $\bar{G}$ and $\bar{H}$ may have sinks (vertices $v$ such that $s^{-1}(v)=\emptyset$).

A first quadrant textile weaved by a textile system $T$ is a two-dimensional array $(e(i,j))\in (G^1)^{{\mathbb N}^2}$, such that  $r(e(i,j-1))=s(e(i,j))$ and such that $q(e(i-1,j))=p(e(i,j))$ for all $i,j\in {\mathbb  N}$. It is clear that  $(e(i,j)) _{j\in {\mathbb N}}\in G^{\infty}$ (the infinite path space of $G$) for all $i\in {\mathbb N}$. In some cases, the set of such arrays may  be empty (see Example 3.1 in \cite{A}).

\remark A textile system associates to each edge  $e\in G^1$ a square called
{\em Wang tile} with bottom edge $s(e)$, top edge $r(e)$, left edge $p(e)$,
and right edge $q(e)$: 
\[\begin{array}{ccc}{}&r(e)&{}\end{array}\]\[\begin{array}{ccc}p(e){}&\begin{tabular}{|c|}\hline e\\ \hline\end{tabular} &{}q(e).\end{array}\]\[
\begin{array}{ccc}{}&s(e)&{}\end{array}\]
If we let $X=X(T)$ to be the set of all textiles
weaved by $T$, then $X$ is a closed, shift invariant subset of
$(G^1)^{{\mathbb N}^2}$, and we obtain a two-dimensional
shift of finite type, defined below. 
Alternatively, if we use Wang tiles, we get a tiling of the first quadrant.
We will describe in Proposition \ref{stot} the connection between two-dimensional shifts of finite type and textile systems.

\definition Let $S$ be a finite alphabet of cardinality $|S|$. The full $d$-dimensional shift with alphabet $S$ is the dynamical system $(S^{{\mathbb N}^d}, \sigma)$, where  \[\sigma^m (x)(n)=x(n+m), \; x\in X,  \; n,m\in {\mathbb N}^d.\]
A subset $X\subset S^{{\mathbb N}^d}$ which is closed in the product topology and which is $\sigma$-invariant 
 is called a   {\em $d$-dimensional shift of finite type} or a  {\em Markov shift} if
there exists a finite set (window) $F\subset {\mathbb N}^d$ and a set of {\em admissible patterns} 
$P\subset S^F$
such  that \[X=X[P]=\{x\in S^{{\mathbb N}^d} \; \mid \;(\sigma^mx)\mid_F\in P\; 
\mbox{for every
}\; m\in {\mathbb N}^d\}.\]
Many times $F=\{(0,0,...,0),(1,0,...,0),(0,1,...,0),...,(0,0,...,1)\}$. A shift of finite type has $d$ transition matrices of dimension $|S|$ with entries in $\{0,1\}$, which in general do not commute. 

\definition\label{conj} Let $S_1$ and $S_2$ be alphabets,  let $F\subset{\mathbb N}^d$ be a finite subset, and let $\Phi:S_1^F\to S_2$ be a map. A sliding block code defined by $\Phi$ is the map \[\phi:S_1^{{\mathbb N}^d}\to S_2^{{\mathbb N}^d}, \phi(x)_n=\Phi(x\mid_{F+n}), n\in {\mathbb N}^d.\]
For $d=1$ we recover the notion of cellular automaton. Two shifts of finite type $X[P_1], X[P_2]$ are {\em conjugate} if there is a bijective sliding block code $\phi:X[P_1]\to X[P_2]$. In this case, the dynamical systems $(X[P_1], \sigma)$ and $(X[P_2], \sigma)$ are topologically conjugate (see \cite{LS}).

For $d=2$, any Markov shift can be specified by two transition matrices. Such shifts are investigated by N.G. Markley and M.E. Paul in \cite{MP1}. Two $k\times k$ transition matrices $A$ and $B$  with no identically zero rows or columns are called {\em coherent} if 
\[(AB)(i,j)>0\;\;\text{iff}\;\; (BA)(i,j)>0\;\;\text{and} \;\;(AB^t)(i,j)>0\;\;\text{iff}\;\; (B^tA)(i,j)>0,\] where $B^t$ is the transpose. If $A$ and $B$ are coherent, it is proved that 
\[X(A,B)=\{x\in S^{{\mathbb N}^2}: A(x(i,j),x(i+1,j))=1\;\; \text{and}\;\; B(x(i,j),x(i,j+1))=1\;\; \text {for all}\;\; (i,j)\in {\mathbb N}^2\}\]
becomes a two-dimensional shift of finite type, where $S=\{0,1,2,...,k-1\}$.

For more about multi-dimensional shifts of finite type, we refer to \cite{S1, S2} and \cite{L, LS}.
We illustrate now with some examples of textile systems and their associated two-dimensional shifts.

\example\label{ex1} 

Let $G^1=\{a,b\}, G^0=\{u,v\}$ with $s(a)=u=r(b), s(b)=v=r(a),$ and let  $H^1=\{x\}, H^0=\{w\}$ with $p(a)=p(b)=x=q(a)=q(b).$ 

\begin{figure}[htbp] 
   \centering
   \includegraphics[width=4in]{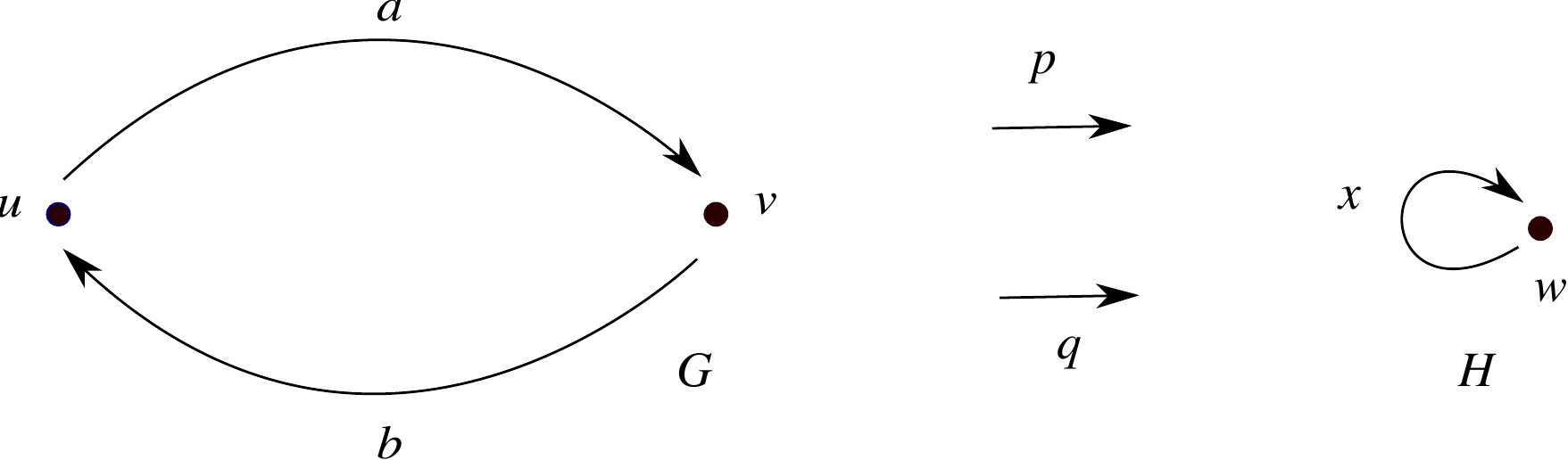} 
   \caption{}
   \label{fig:example1}
\end{figure}

The corresponding two-dimensional shift has alphabet $S=\{a,b\}$ and transition matrices
\[A=\left[\begin{array}{cc}1&1\\1&1\end{array}\right],\; B=\left[\begin{array}{cc}0&1\\1&0\end{array}\right].\]
We will see later that this shift is a particular case of a cellular automaton, obtained from the automorphism of the Bernoulli shift $(\{a,b\}^{\mathbb N}, \sigma)$ which interchanges $a$ and $b$. It also corresponds  to a rank two graph, because the transition matrices commute and the unique factorization property is satisfied (see \cite {KP} section 6).

\example \label{ex2}Let $G^1=\{a,b,c\}, G^0=\{u\}, H^1=\{e,f\}, H^0=\{v\}$ with $p(a)=p(b)=e, p(c)=f, q(a)=f, q(b)=q(c)=e$. 

\begin{figure}[htbp] 
   \centering
   \includegraphics[width=4in]{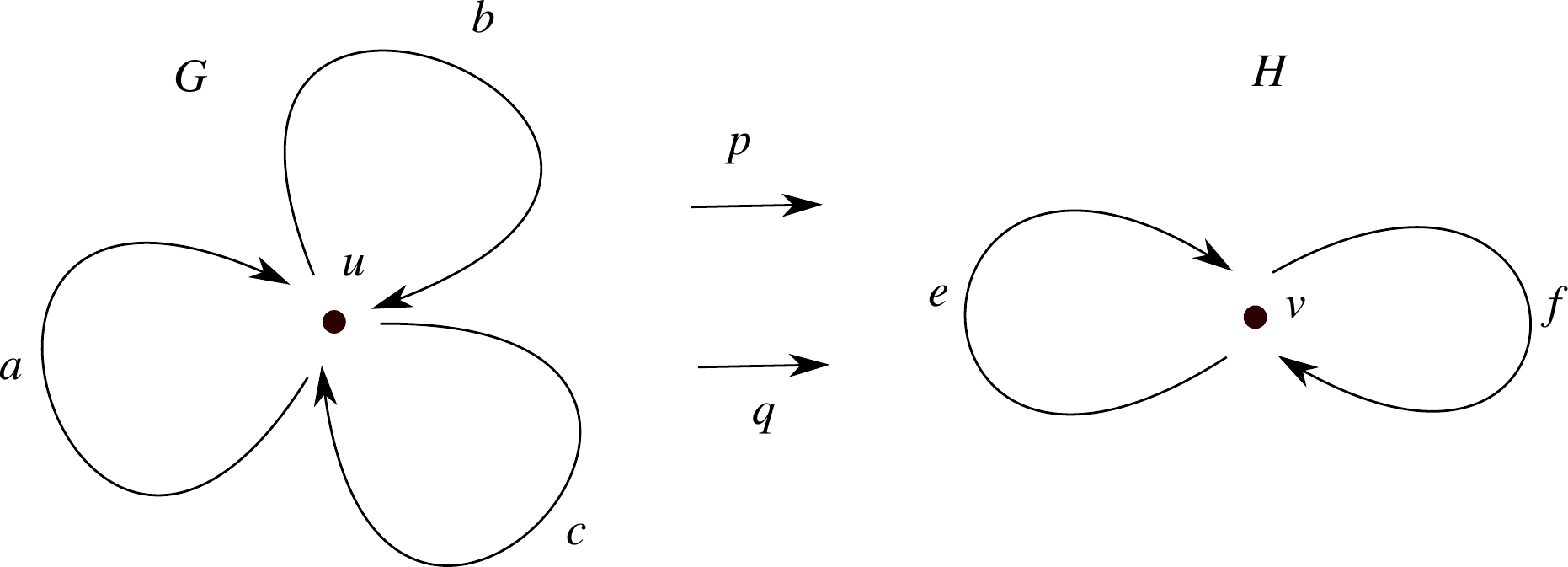} 
   \caption{}
   \label{fig:example}
\end{figure}

Then the corresponding two-dimensional shift of finite type has alphabet $\{a,b,c\}$  and transition matrices
\[A=\left[\begin{array}{ccc}0&0&1\\1&1&0\\1&1&0\end{array}\right],\;\; B=\left[\begin{array}{ccc}1&1&1\\1&1&1\\1&1&1\end{array}\right].\]
Note that $A$ and $B$ are coherent in the sense of Markley and Paul, but do not commute, so this shift is not associated to a rank two graph.

\example \label{ex3} Let $G^1=\{a,b,c\}, G^0=\{u,v\}, s(a)=s(b)=r(c)=r(b)=u, r(a)=s(c)=v, H^1=\{e\}, H^0=\{w\}, p(a)=p(b)=p(c)=q(a)=q(b)=q(c)=e$.

\begin{figure}[htbp] 
   \centering
   \includegraphics[width=4in]{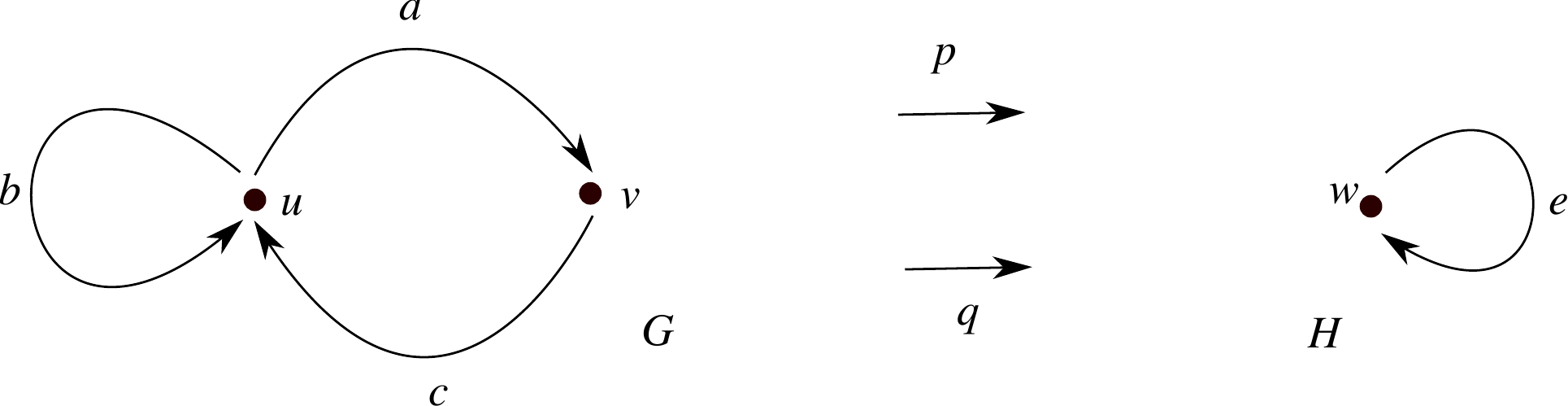} 
   \caption{}
   \label{fig:example}
\end{figure}

This textile system is isomorphic to the dual of the previous one.
The corresponding two-dimensional shift of finite type has the same alphabet, but the transition matrices are interchanged.

\medskip

\section{Textile systems associated to a two-dimensional shift of finite type}

\medskip
From a two-dimensional shift of finite type $X$ we will construct a double sequence of textile systems $T(m,n)$, considering higher  block presentations of $X$ such that $X$ and the shift determined by $T(m,n)$ are conjugated. Recall

\begin{proposition}\label{stot} (see \cite{JM}) Let $(X,\sigma)$ be a two-dimensional shift of finite type with alphabet $S$.  Then, moving to a higher block presentation of $X$ if necessary,  there exists a textile system $T$ such that $X$ is determined by $T$.
\end{proposition}
\begin{proof} Consider ${\mathcal B}={\mathcal B}(2,2)$ the set of $2\times 2$ admissible blocks $\displaystyle \beta=\begin{array}{cc}a&b\\c&d\end{array}$ in $X$, and construct a graph $G$ with $G^0$ labeled by the rows of the blocks in ${\mathcal B}$, $G^1={\mathcal B},\; s(\beta)= c\;\; d$, $r(\beta)=a\;\; b$, and a  graph $H$ with  $H^0=S$ and  $H^1$ labeled by the columns of the blocks in ${\mathcal B}$. Define graph morphisms $p,q:G\rightarrow H$ by $\displaystyle p(\beta)=\begin{array}{c}a\\c\end{array}$, $\displaystyle q(\beta)=\begin{array}{c}b\\d\end{array}$. It is clear that $T=(G,H,p,q)$ is a textile system such that $X$ is the set of textiles weaved by $T$.
\end{proof}

\begin{corollary}\label{ds} For $m,n \ge 1$, let ${\mathcal B}(m,n)$ denote the set of $m\times n$ admissible blocks in $X$, and for $n\ge 2$ define a graph $G(m,n)$ with $G^0(m,n)={\mathcal B}(m,n-1)$ and $G^1(m,n)={\mathcal B}(m,n)$. For $\beta\in G^1(m,n)$, let  $s(\beta)=$ the lower $m\times (n-1)$ block of $\beta$ and let $r(\beta)=$ the upper $m\times (n-1)$ block of $\beta$. Then for $m\ge 2$ there are graph morphisms $p,q:G(m,n)\to G(m-1,n)$ defined by $p(\beta)=$ the left $(m-1)\times n$ block of $\beta$, $q(\beta)=$ the right $(m-1)\times n$ block of $\beta$, where $\beta\in G^1(m,n)$. Then $T(m,n):=(G(m,n), G(m-1,n), p,q)$ for $m,n\ge 2$ are textile systems, and $X$ is determined by $T(m,n)$. The shift $X$ is also determined by the dual textile system  $\bar{T}(m,n):=(\bar{G}(m,n), \bar{G}(m, n-1), s, r)$, where $\bar{G}^1={\mathcal B}(m,n)$, $\bar{G}^0(m,n)={\mathcal B}(m-1,n)$ and the source and range maps are given by $p$ and $q$ as above.
\end{corollary}

We illustrate with some  two-dimensional shifts of finite type and their associated textile systems.  In each case, the 
morphisms $p,q$ are defined as in \ref{stot}.

\example\label{fs}
(The full shift). Let $S=\{0,1\}$ and let $X= S^{{\mathbb N}^2}$. In the corresponding textile system $T=T(2,2)$, the graph $G=G(2,2)$ is the complete graph with $4$ vertices. Indeed,
$\displaystyle G^1=\left\{\begin{array}{cc}a&b\\c&d\end{array}\mid a, b, c, d\in S\right\}$
and 
$G^0=\{0\; 0, 0\; 1, 1\; 0, 1\; 1\}.$ 
The graph $H=G(1,2)$ is the complete graph with $2$ vertices. Indeed,
$\displaystyle H^1=\left\{\begin{array}{c}0\\0\end{array}, \begin{array}{c}1\\0\end{array}, \begin{array}{c}0\\1\end{array}, \begin{array}{c}1\\1\end{array}\right\}$
and  $H^0=\{0,1\}$. 

\example\label{L} (Ledrappier). Let $S={\mathbb Z}/2{\mathbb Z}$,  and let  $X\subset S^{{\mathbb N}^2}$ be the subgroup defined by $x\in X$ iff  \[x(i+1,j)+x(i,j)+x(i,j+1)=0\;\;\text{for all}\;\; (i,j)\in{\mathbb N}^2.\] 
We have $G^0(2,2)=H^1=S\times S$, and $G^1(2,2)$ has 8 elements, corresponding
to the $2\times 2$ matrices $(a(i,j))$ with entries in $S$ such that $a(1,1)+a(2,1)+a(2,2)=0$. The Ledrappier shift is associated  to a rank two graph, and if we  consider the new alphabet 
\[\begin{array}{cc}0&{}\\0&0\end{array},\;\; \begin{array}{cc}1&{}\\0&1\end{array},\;\; \begin{array}{cc}1&{}\\1&0\end{array}, \;\; \begin{array}{cc}0&{}\\1&1\end{array},\]
then  the transition matrices are
\[A=\left[\begin{array}{cccc}1&1&0&0\\0&0&1&1\\1&1&0&0\\0&0&1&1\end{array}\right],\;\; B=\left[\begin{array}{cccc}1&1&0&0\\0&0&1&1\\0&0&1&1\\1&1&0&0\end{array}\right],\]
 see \cite{PRW1}.

\example\label{gm} (Golden Mean). Let  $S=\{0,1\}$,  with transition matrices 
\[A=B= \left[\begin{array}{cc}1&1\\1&0\end{array}\right].\]
Then in the corresponding textile system $T=T(2,2)$, the graphs $G=G(2,2)$ and $H=G(1,2)$ have
\[G^0
=\{0\; 0,0\; 1,1\; 0\},\]
\[G^1=\left\{\begin{array}{cc}0&0\\0&0\end{array},\begin{array}{cc}0&1\\0&0\end{array}, \begin{array}{cc}0&0\\0&1\end{array}, \begin{array}{cc}1&0\\0&0\end{array}, \begin{array}{cc}1&0\\0&1\end{array}, \begin{array}{cc}0&0\\1&0\end{array}, \begin{array}{cc}0&1\\1&0\end{array}\right\},\] \[ H^0=\{0,1\},\;\;H^1=\left\{\begin{array}{c}0\\0\end{array},\begin{array}{c}1\\0\end{array},\begin{array}{c}0\\1\end{array}\right\}.\]

\example\label{ca} (Cellular automata). Let $k\geq 1$ and let $Y\subset \{0,1,...,k-1\}^{\mathbb N}$ be a subshift of finite type. It is known that a
continuous, shift-commuting onto map $\varphi :Y\rightarrow Y$ is given by a sliding block code. Given such a $\varphi$, define a closed,
shift invariant subset \[X=\{(y_m)\in Y^{\mathbb N}\;\mid \; y_{m+1}=\varphi(y_m)\;\mbox{for all}\;
m\in {\mathbb N}\}\subset \{0,1,...,k-1\}^{{\mathbb N}^2}.\]In a natural way,
 $X$ becomes a two-dimensional Markov shift. In the corresponding textile system $T=T(2,2)$, we have $G^0\subset \{0,1,...,k-1\}\times \{0,1,...,k-1\}$, $G^1=$ the set of admissible $2\times 2$ blocks $\begin{array}{cc}a&b\\c&d\end{array}$ with $a,b,c,d\in \{0,1,...,k-1\}$, $H^0=\{0,1,...,k-1\}$, and $H^1=$ the set of admissible columns $\begin{array}{c}a\\c\end{array}$.
 
 For $k=2$, $Y=\{0,1\}^{\mathbb N}$ and $\varphi$ defined by interchanging the letters $0$ and $1$, we recover the textile system from example \ref{ex1}.

Recall that many rank two graphs can be  obtained from two finite graphs $G_1$ and $G_2$ with the same set of vertices  such that the associated vertex matrices commute, and a fixed bijection $\theta: G_1^1*G_2^1\to G_2^1*G_1^1$ such that if $\theta(\alpha,\beta)=(\beta',\alpha')$, then $r(\alpha)=r(\beta')$ and $s(\beta)=s(\alpha')$. Here
\[G_1^1*G_2^1:=\{(\alpha,\beta)\in G_1^1\times G_2^1\mid\; s(\alpha)=r(\beta)\},\]
and $s, r$ are the source and range maps.
This rank two graph is denoted by $G_1*_{\theta}G_2$. The infinite path space is a first quadrant grid with horizontal edges from $G_1$ and vertical edges from $G_2$. Each $1\times 1$ square is uniquely determined by one horizontal edge followed by one vertical edge. 

\begin{proposition}\label{gtot} Any rank two graph of the form $G_1*_{\theta}G_2$ determines a textile system. 
\end{proposition}
\begin{proof}
Indeed, let $H_i=G_i^{op}$, the graph $G_i$ with the source and range maps interchanged for $i=1,2$. The map $\theta$ induces a unique bijection $H_1^1*H_2^1\to H_2^1*H_1^1$, where
\[H_1^1*H_2^1=\{(\alpha,\beta)\in H_1^1\times H_2^1\mid\; r(\alpha)=s(\beta)\}.\]
We let $G$  with $G^1=H_1^1*H_2^1$ identified with $H_2^1*H_1^1$ by the map $\theta$, $G^0=H_1^1$, and we let $H=H_2$. Define  $s(\alpha,\beta)=\alpha,\;\; r(\alpha,\beta)=\alpha',\;\; p(\alpha,\beta)=\beta'$, and $q(\alpha,\beta)=\beta$, where $\alpha', \beta'$ are uniquely determined by the bijection $\theta(\alpha,\beta)=(\beta',\alpha')$. 
\end{proof}
 
 \remark For a cellular automaton with $\varphi$ as in \ref{ca} defined by an automorphism of a rank one graph $G$, in \cite{FPS} the authors associated a rank two graph whose $C^*$-algebra is a crossed product $C^*(G)\rtimes{\mathbb Z}$, and they computed its K-theory.

\medskip

\section{$C^*$-algebras associated to a two-dimensional shift of finite type}

\medskip

Recall that in Corollary \ref{ds} we constructed a family $T(m,n)=(G(m,n), G(m-1,n),p,q)$ of textile systems from a two-dimensional shift of finite type. This defines a family of graph $C^*$-algebras ${\mathcal A}(m,n):=C^*(G(m,n))$ for $m,n\ge 2$. The dual textile system $\bar{T}(m,n)=(\bar{G}(m,n), \bar{G}(m, n-1), s, r)$ determines another family $\bar{\mathcal A}(m,n):=C^*(\bar{G}(m,n))$, where $\bar{G}(m,n)$ is the graph with source and range maps given by $p$ and $q$, described in Corollary \ref{ds}.

\remark  We have ${\mathcal A}(m,n)\cong {\mathcal A}(m,2)$ for all $n\ge 2$ and $\bar{\mathcal A}(m,n)\cong \bar{\mathcal A}(2,n)$ for all $m\ge 2$. Indeed, the graph $G(m,n)$ is a higher block presentation of $G(m,2)$ and the graph $\bar{G}(m,n)$ is a higher block presentation of $\bar{G}(2,n)$ (see \cite{B}).

For matrix subshifts, we can be more specific. Consider $A,B$ two coherent $k\times k$ transition matrices indexed by $\{0,1,...,k-1\}$ as in \cite{MP2}, and let $X(A,B)$ be the  associated matrix shift. 

\begin{theorem}\label{t1} For a matrix shift  $X(A,B)$ we have  $\bar{\mathcal A}(2,n)\cong{\mathcal O}_{A_n}$  and ${\mathcal A}(n,2)\cong{\mathcal O}_{B_n}$ for $n\ge 2$.  The transition matrices $A_n$ and $B_n$ can be constructed inductively as in \cite{MP2}, and they define two sequences $(Y(A_n))_{n\ge 1}$ and $(Y(B_n))_{n\ge 1}$ of  one-dimensional shifts of finite type associated to $X(A,B)$.   
\end{theorem} 

\begin{proof}

Consider the  strip \[K_n= \{(i,j)\in {\mathbb N}^2: 0\le j\le n-1\}\] and the alphabet \[Q_n={\mathcal B}(1,n)=\{\alpha:\alpha\; \text{is a}\; 1\times n\; \text{block occuring in}\; X(A,B)\},\] ordered lexicographically starting at the top. Define $Y(A_n)=\{x\mid_{K_n}: x\in X(A,B)\}$ to be the Markov shift with alphabet $Q_n$ and transition matrix $A_n$, obtained by restricting elements of $X(A,B)$ to the strip $K_n$. The shift $Y(B_n)$ is defined similarly, considering strips \[L_n=\{(i,j)\in {\mathbb N}^2: 0\le i\le n-1\}\] and alphabets \[R_n={\mathcal B}(n,1)=\{\beta:\beta \; \text{is a}\; n\times 1\; \text{block occuring in}\; X(A,B)\}.\]

Clearly, $A_1=A$ and $B_1=B$.  For $n\ge 2$, $A_n$ is a $k_n\times k_n$ matrix, where $k_n$ is the sum of all entries in $B^{n-1}$. Suppose $\alpha$ and $\alpha'$ are $1\times n$ blocks in $X(A,B)$ and $j,j'\in \{0,1,2,...,k-1\}$. Then 
\[A_{n+1}\left(\begin{array}{c} j\\\alpha\end{array}, \begin{array}{c}j'\\\alpha'\end{array}\right)=1\]
if and only if the $1\times (n+1)$ blocks $\begin{array}{c} j\\\alpha\end{array}$ and $\begin{array}{c}j'\\\alpha'\end{array}$ occur in $X(A,B)$ and $A(j,j')A_n(\alpha,\alpha')=1$.  By Proposition 2.1 in \cite{MP2}, the matrix $A_{n+1}$ is the principal submatrix of $A\otimes A_n$ obtained by deleting the $m$th row and column of $A\otimes A_n$ if and only if $B(i,j)=0$, where $m=jk_n+h,\;\; 0\le h< k_n$, and 
\[\sum_{l=0}^{i-1}\sum_{t=0}^{k-1}B^{n-1}(t,l)\le h< \sum_{l=0}^{i}\sum_{t=0}^{k-1}B^{n-1}(t,l).\]
The matrix $B_{n+1}$ is constructed similarly, by deleting rows and columns from $B\otimes B_n$.

\end{proof}

Recall that the dynamical system $(X(A,B), \sigma)$ is (topologically) strong mixing if given any nonempty open sets $U$ and $V$ in $X(A,B)$, there is $N\in{\mathbb N}^2$ such that $\sigma^n(U)\cap V\neq \emptyset$ for all $n\ge N$ (componentwise order).

\begin{corollary} Assume that the transition matrices $A_n, B_n$ are not permutation matrices. Then the $C^*$-algebras $\bar{\mathcal A}(2,n)$ and ${\mathcal A}(n,2)$ are simple and purely infinite if and only if $(X(A,B),\sigma)$ is strong mixing.
\end{corollary}
\begin{proof} Apply Proposition 2.2 in \cite{MP2}.
\end{proof}

\example\label{fs1} For the full shift described in \ref{fs}, we have $A=B=\left[\begin{array}{cc}1&1\\1&1\end{array}\right]=A_1=B_1$ and $A_{n+1}=B_{n+1}=A\otimes A_n$ for  $n\ge 1$. The corresponding $C^*$-algebras are $\bar{\mathcal A}(2,n)\cong{\mathcal A}(n,2)\cong{\mathcal O}_{2^n}$.

\example\label{ex1a} For the shift associated to the textile system in \ref{ex1}, we have 
\[A_1=A=\left[\begin{array}{cc}1&1\\1&1\end{array}\right], \;\; A_2=\left[\begin{array}{cccc}1&1&1&1\\1&1&1&1\\1&1&1&1\\1&1&1&1\end{array}\right], ...\]
\[B_1=B=\left[\begin{array}{cc}0&1\\1&0\end{array}\right], \;\; B_2=\left[\begin{array}{cccc}0&0&0&1\\0&0&1&0\\0&1&0&0\\1&0&0&0\end{array}\right],...\]
with corresponding sequences of $C^*$-algebras $\bar{\mathcal A}(2,n)\cong{\mathcal O}_{2^n}$  and ${\mathcal A}(n,2)\cong C({\mathbb T})\otimes M_{2^n}$, since $A_n$ is the $2^n\times 2^n$ matrix with all entries $1$, and $B_n$ is a $2^n\times 2^n$ permutation matrix. Note that the dynamical system $(X(A,B),\sigma)$ is not strong mixing.

\example\label{gm1} Consider the Golden Mean shift $X(A,A)$ with $A=\left[\begin{array}{cc}1&1\\1&0\end{array}\right]$. Since the number of $n$-words in $Y(A)$ is a Fibonacci number, the dimension $k_n$ of $A_n=B_n$ is also a Fibonacci number, where $k_1=2$ and $k_2=3$. It is easy to see that to get $A_{n+1}$ from $A\otimes A_n$ we have to remove the last $2k_n-k_{n+1}$ rows and columns. Thus
\[A_1=A,\;\; A_2=\left[\begin{array}{ccc}1&1&1\\1&0&1\\1&1&0\end{array}\right],\;\; A_3=\left[\begin{array}{ccccc}1&1&1&1&1\\1&0&1&1&0\\1&1&0&1&1\\1&1&1&0&0\\1&0&1&0&0\end{array}\right],\]\[ A_4=\left[\begin{array}{cccccccc}1&1&1&1&1&1&1&1\\1&0&1&1&0&1&0&1\\1&1&0&1&1&1&1&0\\1&1&1&0&0&1&1&1\\1&0&1&0&0&1&0&1\\1&1&1&1&1&0&0&0\\1&0&1&1&0&0&0&0\\1&1&0&1&1&0&0&0\end{array}\right]\] etc, which are transitive and not permutation matrices.
The sequence of simple purely infinite $C^*$-algebras ${\mathcal A}(2,n)\cong\bar{\mathcal A}(n,2)\cong{\mathcal O}_{A_n}$ encodes the complexity of the Golden Mean shift.

\remark We have natural projections $Y(A_{n+1})\to Y(A_n)$ and $Y(B_{n+1})\to Y(B_n)$ such that
\[X(A,B)=\varprojlim Y(A_n)=\varprojlim Y(B_n).\]
The families of $C^*$-algebras ${\mathcal A}(m,n)$ and $\bar{\mathcal A}(m,n)$ can be thought as  $C^*$-bundles over $\{(m,n)\in{\mathbb N}^2\mid \;\; m,n\ge 2\}$, and we can interpret the corresponding section $C^*$-algebras  as other algebras associated to  the shift $X(A,B)$. In the case $X(A,B)$ is constructed from a rank two graph, the relationship between the graph $C^*$-algebra and the above $C^*$-algebras  remains to be explored.
\medskip

\section{Groupoid morphisms and Fell bundles from textile systems}

\medskip

\definition  A surjective graph morphism $\phi:G\to H$
has the path lifting property for $s$ (or $\phi$ is an $s$-fibration)  if  for all $v\in G^0$ and  for all $b\in H^1$ with $s(b)=w=\phi(v)$ there is  $a\in G^1$ with $s(a)=v$  with $\phi(a)=b$.  Similarly, we define an $r$-fibration. If the morphism $\phi$ has the path lifting property for both $s$ and $r$, we say that $\phi$ is a fibration. The morphism $\phi$ is a covering if it has the {\em unique} path lifting property for both $s$ and $r$.

\begin{remark} The morphisms $p$ and $q$ in the textile systems from examples \ref{ex2} and \ref{ex3} are fibrations. The morphism $p=q$ in example \ref{ex1} is a covering. The canonical morphisms $p$ and $q$ for the full shift (see \ref{fs}) are covering maps, but the full shift does not define a rank two graph, because the unique factorization property fails. Also, note that in this case, the horizontal and vertical shifts are not local homeomorphisms.

In general, the morphisms  $p$ and $q$ in a textile system don't have the path lifting property: let $G^1=\{a,b,c\}, G^0=\{u,v\}, s(a)=r(a)=u, s(b)=r(c)=u, s(c)=r(b)=v, H^1=\{e,f\}, H^0=\{w\}, p(u)=p(v)=q(u)=q(v)=w, p(a)=p(b)=e, p(c)=f, q(a)=e, q(b)=q(c)=f$.

\begin{figure}[htbp] 
   \centering
   \includegraphics[width=4in]{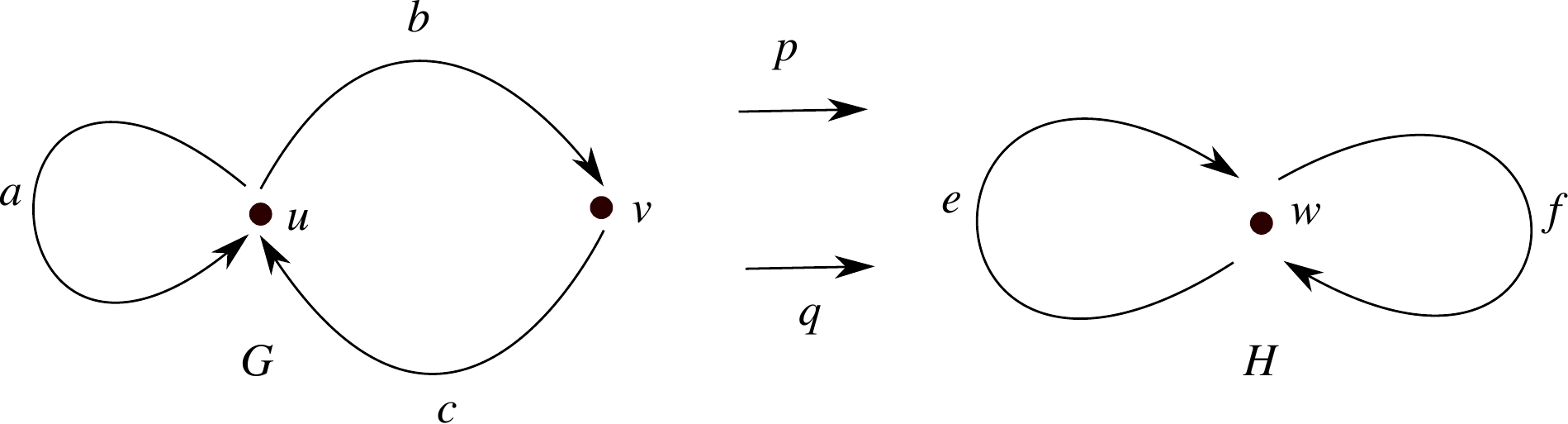} 
   \caption{}
   \label{fig:example}
\end{figure}

Then for $u\in G^0$ and $f\in H^1$ with $s(f)=p(u)=w$ there is no edge $x\in G^1$ with $s(x)=u$ and $p(x)=f$. Also, for $v\in G^0$ and $e\in H^1$ with $s(e)=w=q(v)$ there is no $x\in G^1$ with $s(x)=v$ and $q(x)=e$.  Note also that the graph $\bar{G}=(G^1,H^1)$ from the dual textile system has sinks.
\end{remark}

\begin{proposition} Consider any rank two graph of the form $G_1*_{\theta}G_2$ with the corresponding textile system described in Proposition \ref{gtot}. Then the morphism $q$ has the unique path lifting property for $s$, and the morphism $p$ has the unique path lifting property for $r$.
\end{proposition}
\begin{proof} Indeed, given $\alpha\in G^0=H_1^1$ and $\beta\in H^1=H_2^1$ with $q(\alpha)=s(\beta)$, there is a unique $(\alpha,\beta)\in G^1=H_1^1*H_2^1$ such that $s(\alpha, \beta)=\alpha$ and $q(\alpha,\beta)=\beta$. The proof for $p$ is similar.
\end{proof}

\begin{remark} For the textile system $T(2,2)$ associated to a two-dimensional shift, we can characterize the (unique) path lifting property  for the morphisms $p,q$ in terms of filling a corner of a $2\times 2$ block. For example, $p$ has the (unique) path lifting property for $s$ if for any admissible column
$\displaystyle\begin{array}{c}a\\c\end{array}$ and for any admissible row $c\;\; d$, there is a (unique) $b$ which completes the admissible block $\displaystyle \beta=\begin{array}{cc}a&b\\c&d\end{array}.$ Similarly, we can characterize the path lifting property for the morphisms $p,q$ in $T(m,n)$. 
\end{remark}

\medskip
For a topological groupoid $\Gamma$, we denote by $s$ and $r$ the source and the range maps, by $\Gamma^0$ the unit space, and by $\Gamma^2$ the set of composable pairs.
\definition 
Let $\Gamma, \Lambda$  be  topological groupoids. 
A {\em groupoid morphism} $\pi :\Gamma\rightarrow \Lambda$ is a
continuous map which
intertwines both the range and source maps and which satisfies
\[\pi(\gamma_1\gamma_2)=\pi(\gamma_1)\pi(\gamma_2)\;\; \text{for all}\; \;(\gamma_1,\gamma_2)\in \Gamma^2.\]  It follows that \[\ker\pi:=\{\gamma\in \Gamma\mid \pi(\gamma)\in \Lambda^0\}\] contains the unit space $\Gamma^0$. A {\em groupoid fibration} is a surjective open morphism
 $\pi :\Gamma\rightarrow \Lambda$ such
that for any $\lambda\in \Lambda$ and $x\in \Gamma^0$ with $\pi(x)=s(\lambda)$ there is $\gamma\in
\Gamma$ with $s(\gamma)=x$ and $\pi(\gamma)=\lambda$. Note that, using inverses, a groupoid fibration also has the property that for any $\lambda\in \Lambda$ and $x\in \Gamma^0$ with $\pi(x)=r(\lambda)$ there is $\gamma\in
\Gamma$ with $r(\gamma)=x$ and $\pi(\gamma)=\lambda$. If $\gamma$ is unique, then $\pi$ is called a {\em groupoid covering}.

For $G$ a finite graph without sinks, let  $G^\infty$ be the space of infinite paths, and let $\sigma:G^\infty\to G^\infty$ be the unilateral shift $\sigma(x_1x_2x_3\cdots)=x_2x_3\cdots$. Let
\[\Gamma(G)=\{(x,m-n,x')\in G^{\infty}\times {\mathbb  Z}\times G^{\infty}\; \mid \;  \sigma^m(x)=\sigma^n(x')\}\]
be the corresponding \' etale groupoid with unit space $\Gamma(G)^0=\{(x,0,x)\mid x\in G^{\infty}\}$ identified with  $G^{\infty}$.

\proposition  Let $G, H$ be finite graphs with no sinks. Then any   morphism  $\phi:G\to H$  with the path lifting property for $s$  induces a
surjective continuous open map 
\[\varphi: G^{\infty}\rightarrow H^{\infty}, \quad\varphi(x_1x_2x_3
\cdots)=\phi(x_1)\phi(x_2)\phi(x_3)\cdots\] 
 and  a groupoid fibration \[\pi: \Gamma(G)\rightarrow \Gamma(H), 
\quad\text{given by}\quad
\pi(x,k,x')=(\varphi(x),k,\varphi(x'))\] 
 with kernel 
$\Delta=\{(x,0,x')\in \Gamma(G)\;\mid\; \varphi(x)=\varphi(x')\}.$ If $\phi$ is a graph covering, then $\pi$ is a groupoid covering.

\begin{proof}  Let $y_1y_2\cdots \in H^{\infty}$ beginning
 at $w_1\in H^0$. Since $\phi$ is onto, there is $v_1\in G^0$ with $\phi(v_1)=w_1$. By the path lifting property, there is $x_1\in G^1$ with $\phi(x_1)=y_1$. Continuing inductively, there is $x_1x_2\cdots\in G^{\infty}$ such that 
 $\varphi(x_1x_2\cdots)=y_1y_2\cdots$, and therefore $\varphi$ is onto.
 Consider a cylinder set  \[Z=\{a_1\cdots a_nx_1x_2\cdots \in G^{\infty}\;\mid x_1x_2\cdots\in G^{\infty}\}.\]  By the path lifting property, $\varphi(Z)$ is the cylinder set in $H^{\infty}$ determined by the finite path $\phi(a_1)\cdots\phi(a_n)$. Hence $\varphi:G^{\infty}\to H^{\infty}$ is continuous and open. 
 We have \[\pi((x,k,x')(x',l,x''))=\pi(x,k,x')\pi(x',l,x'')=(\varphi(x),k+l,\varphi(x'')),\] and $\pi$  is a groupoid morphism.

 Since $\varphi$ is surjective and takes cylinder sets into cylinder sets,  $\pi$ is surjective, continuous and open. 
 To show that $\pi$ is a  fibration, consider $\lambda=(y,k,y')\in \Gamma(H)$ and $x'\in \Gamma(G)^0=G^{\infty}$ with $\varphi(x')=s(\lambda)=y'$. Since $\varphi$ is onto and intertwines   the shift maps, we can find $\gamma=(x,k,x')\in \Gamma(G)$ with $\pi(\gamma)=\lambda$. Hence $\pi$ is a groupoid fibration. In the case $\phi$ is a covering, let's show how we can find $\gamma$ in a  unique way. We have $k=m-n$ and $\sigma^my=\sigma^ny', \sigma^mx=\sigma^nx'$. For $y_m$ and $v=s(x_{m+1})=s(x'_{n+1})$ with $r(y_m)=\phi(v)$ there is a unique $x_m$ with $r(x_m)=v$ and $\phi(x_m)=y_m$. We can continue inductively to find a unique $x$ with $\varphi(x)=y$, and it follows that $\pi$ is a groupoid covering.
Now $(x,k,x')\in \ker\pi$ iff $\varphi(x)=\varphi(x')$ and $k=0$. \end{proof}

\corollary Given a textile system $(G,H,p,q)$ such that $G, H$ have no sinks and $p, q$ have the path lifting property, we get two groupoid fibrations $\pi, \rho :\Gamma(G)\to \Gamma(H)$. If $p$ and $q$ are coverings, we get two groupoid coverings $\pi, \rho :\Gamma(G)\to \Gamma(H)$.

\example\label{fs2} Consider the coverings $p=q:G\to H$ in the textile system of the full shift as in Example \ref{fs}.  We obtain a  covering $\pi=\rho:\Gamma(G)\to\Gamma(H)$ of Cuntz groupoids. 

\medskip

Recall that a (saturated) Fell bundle over a groupoid $\Gamma$ is a Banach bundle
  $\pi: E \to \Gamma$ with extra structure such that the fiber $E_{\gamma}=\pi^{-1}(\gamma)$ is an $E_{r(\gamma)}$--$E_{s(\gamma)}$
imprimitivity bimodule for all $\gamma\in\Gamma$. The restriction of $E$ to the unit space $\Gamma^0$ is a
$C^*$-bundle.
The $C^*$-algebra $C^*_r(\Gamma; E)$ is a completion of $C_c(\Gamma; E)$ in ${\mathcal L}(L^2(\Gamma; E))$. For more details, see \cite{DKR}, where
 the following result is proved.

\begin{theorem} Given an open surjective morphism of \'etale  groupoids
$\pi:\Gamma\rightarrow \Lambda$ with amenable kernel $\Delta := \pi^{-1}(\Lambda^0)$, there is
 a Fell bundle $E=E(\pi)$ over $\Lambda$ such that
$C^*_r(\Gamma)\cong C^*_r(E)$. 
\end{theorem}
Using Corollary \ref{ds} and  Theorem \ref{t1}, we get 

\begin{theorem}\label{t2} Given a matrix shift $X(A,B)$ such that in the associated family of textile systems $T(m,n)$ the morphisms $p$ and $q$ have the path lifting property, there are two families of Fell bundles $E^{(m,n)}(p)$ and $E^{(m,n)}(q)$ over $\Gamma(G(m-1,n))$ such that
\[{\mathcal A}(m,n)\cong C^*_r(E^{(m,n)}(p))\cong C^*_r(E^{(m,n)}(q)).\]
\end{theorem}

\example\label{gr}

Consider the textile system $(G,H,p,q)$ from example \ref{ex1}.
In this case
 $\Gamma(G)^0$ has  two points, and $C^*(G)\cong  C({\mathbb T})\otimes M_2$. The maps $p$ and $q$ are coverings, they both induce the morphism \[\pi:\Gamma(G)\to \Gamma(H)\cong{\mathbb Z}, \;\pi(x,k,x')=k,\] and the two Fell bundles $E^{(2,2)}(p)$ and $E^{(2,2)}(q)$ over  ${\mathbb Z}$ coincide. The fiber over $0\in {\mathbb Z}$ is isomorphic to $M_2$. 
 
 \begin{remark} For a rank two graph $G_1*_{\theta}G_2$, since the map $q$ in the corresponding textile system is a covering, we get a goupoid covering $\pi:\Gamma(G)\to\Gamma(H)$. Recall that $G^1=H_1^1*H_2^1, G^0=H_1^1$ and $H=H_2$, where $H_i=G_i^{op}$. In particular, $\Gamma(H)$ acts on $\Gamma(G)^0$ and $\Gamma(G)\cong \Gamma(H)\ltimes \Gamma(G)^0$ (see \cite{DKR} Proposition 5.3).
 \end{remark}

\example  
Consider the textile system from example \ref{ex2}.
Here $G^{\infty}=\{a,b,c\}^{\mathbb N}, H^{\infty}=\{e,f\}^{\mathbb N}$ are Cantor sets, $C^*(G)\cong  {\mathcal O}_3$ and $C^*(H)\cong{\mathcal O}_2$, the Cuntz algebras. The morphisms $p$ and $q$ are fibrations and induce different groupoid morphisms $\pi, \rho:\Gamma(G)\to \Gamma(H)$. The fibers of the Fell bundle $E^{(2,2)}(p)$ over $y\in \Gamma(H)^0=H^{\infty}$
 are isomorphic to  $M_{2^n}$, where $n$
is the number of $e's$ in $y$. For $n=\infty$, $M_{2^{\infty}}$ is the
UHF-algebra of type $2^{\infty}$.

\example Let $(G,H,p,q)$ be the textile system from example \ref{ex3}. 
The space $G^{\infty}\subset \{u,v\}^{\mathbb N}$ is defined by the vertex matrix
\[A=\left[\begin{array}{cc}1&1\\1&0\end{array}\right]\]
and $C^*(G)\cong {\mathcal O}_A$. The space $H^{\infty}$ has one point, and $\Gamma(H)\cong {\mathbb Z}$. Both $p$ and $q$ induce the same morphism $\pi:\Gamma(G)\to {\mathbb Z}, \;\pi(x,k,x')=k$ as in  Example \ref{gr}, and the  Fell bundle $E^{(2,2)}(p)=E^{(2,2)}(q)$  corresponds to the grading of ${\mathcal O}_A$.

\example The full shift $X= \{0,1\}^{{\mathbb N}^2}$ determines   a sequence of textile systems $T(n,2)=\bar{T}(2,n)$, where $G(n,2)=\bar{G}(2,n)$ is the complete graph with $2^n$ vertices, and $G(n-1,2)=\bar{G}(2,n-1)$ is the complete graph with $2^{n-1}$ vertices. The two families of Fell bundles over the Cuntz groupoid $\Gamma(G(n-1,2))$ have C*-algebras  isomorphic to $C^*(G(n,2))\cong {\mathcal O}_{2^n}$. 

\example For the Golden Mean shift with transition matrices
\[A=B= \left[\begin{array}{cc}1&1\\1&0\end{array}\right],\]
the corresponding graphs $G(n,2)=\bar{G}(2,n)$ have vertex matrices as in Example \ref{gm1}.

The morphisms $p$ and $q$ in $T(2,n)=\bar{T}(n,2)$ are fibrations and determine different groupoid morphisms $\Gamma(G(2,n))\to \Gamma(G(2,n-1))$ and two Fell bundles $E^{(2,n)}(p)$ and $E^{(2,n)}(q)$ over $\Gamma(G(2,n-1))$.
It would be interesting to calculate the fibers of these Fell bundles.

\end{document}